\documentclass[preprint,review,10pt]{amsart}
\usepackage{amsmath,amssymb,amsfonts,xspace}
\newtheorem{theorem}{Theorem}[section]
\newtheorem{lemma}[theorem]{Lemma}

\theoremstyle{definition}
\newtheorem{definition}[theorem]{Definition}
\newtheorem{example}[theorem]{Example}

\newtheorem{proposition}[theorem]{Proposition}
\newtheorem{corollary}[theorem]{Corollary}

\theoremstyle{remark}
\newtheorem{remark}[theorem]{Remark}

\numberwithin{equation}{section}

\begin{document}

\title{  $(u,v)$-absorbing primary hyperideals in   multiplicative hyperrings }

\author{Mahdi Anbarloei}
\address{Department of Mathematics, Faculty of Sciences,
Imam Khomeini International University, Qazvin, Iran.
}

\email{m.anbarloei@sci.ikiu.ac.ir}





\begin{abstract}
The present paper addresses the notion of $(u,v)$-absorbing  primary hyperideals in commutative multiplicative hperrings. 

\end{abstract}
\maketitle

\section{Introduction}
The notions of prime and primary ideals take an important place  in ring theory. Many papers have been written on these concepts and their extensions. A recent paper \cite{Khali} is devoted to the study of notion of $(u,v)$-absorbing  primary ideals. A proper ideal $I$ of a commutative ring $R$ refers to  a $(u,v)$-absorbing  primary ideal if    $x_1  \cdots  x_u \in  P$  where $x_1,\ldots,x_u$ are nonunit elements in $R$, then  $x_1  \cdots  x_v \in I$ or $x_{v+1}  \cdots  x_u \in rad(I)$.

The idea of algebraic hyperstructures as a well established branch of classical algebraic theory goes back to
Marty's research work \cite{marty} presented  at the $8^{th}$ Congress of Scandinavian Mathematicians in 1934. So far, many mathematicians has studied on hyperstructures  (\cite{f1}, \cite{f2}, \cite{f4},\cite{f5},\cite{x2},\cite{x1},\cite{f11},\cite{f7},\cite{f17},\cite{f8},\cite{f9}). The hyerrings as a  class of algebraic hyperstructures were introduced and studied by many authors. In 1983, multiplicative hyperrings as  significant class of hyperrings were presented by Rota \cite{f14}. In these hyperrings, the addition is an operation and the multiplication is a hyperoperation.  This type of hyperstructurs  has been widely studied and investigated in \cite{ameri}, \cite{ameri5}, \cite{ameri6},  \cite{Ciampi},  \cite{Kamali} and \cite{f16}. More exactly, a  hyperoperation $``\circ" $ on a non-empty set $A$ is a mapping from $A \times A$ into  the family of all non-empty subsets of $A$ denoted by$P^*(A)$.  If $``\circ" $ is a hyperoperation on  $A$, then we say that $(A,\circ)$  is a  hypergroupoid \cite{f10}.  For any given element $a \in A$ and subsets  $A_1$ and $A_2$ of $A$,  $A_1 \circ A_2 =\cup_{a_1 \in A_1, a_2 \in A_2}a_1 \circ a_2,$ and $ A_1 \circ a=A_1 \circ \{a\}.$  The  hypergroupoid $(A, \circ)$ is   a semihypergroup if $\circ$ is associative that is $\cup_{a \in y \circ z}x \circ a=\cup_{b \in x \circ y} b \circ z $ for any $x,y,z \in A$. The semihypergroup $A$ is  a hypergroup if  $A\circ a=a \circ A=A$ for any  $a \in A$ \cite{f10}. Let $(A,\circ)$ be a semihypergroup. $\varnothing \neq B \subseteq A$ is a subhypergroup if   $B \circ a=a \circ B=B$ for each $a \in B$ \cite{f10}.    If  {\bf 1.} $(A,+)$ is a commutative group,  {\bf 2.} $(A,\circ)$ is a semihypergroup; 
{\bf 3.} $(-x)\circ y =x\circ (-y) =  -(x\circ y)$ for any $x, y \in A$,
{\bf 4.}  $(y+z)\circ x \subseteq y\circ x+z\circ x$ and  $x\circ (y+z) \subseteq x\circ y+x\circ z$ for any $x, y, z \in A, $
{\bf 5.} $x \circ y =y \circ x$ for any $x,y \in A$, then the triple  $(A,+,\circ)$ refers to a  commutative multiplicative hyperring  \cite{f10}. The multiplicative hyperring $A$ is  strongly distributive if in (4), the equality holds.

 For each subset $\Phi \in P^\star(\mathbb{Z})$   where $(\mathbb{Z},+,\cdot)$ is the ring of integers  and $\vert \Phi \vert \geq 2$, there exists a multiplicative hyperring $(\mathbb{Z}_\Phi,+,\circ)$ such that  $\mathbb{Z}_{\Phi}=\mathbb{Z}$ and  $a \circ b =\{a.x.b\ \vert \ x \in \Phi\}$ for all $a,b\in \mathbb{Z}_\Phi$ \cite{das}.
An element $e$ in $ A$ is considered as an identity element if $a \in e\circ a$ for every  $a \in A$ \cite{ameri}. An element $x$ in $ A$ is called unit, if there exists $y $ in $ A$ such that $e \in y \circ x$.  We denote the set of all unit elements in $A$ by $U(A)$  \cite{ameri}. Furthermore, a hyperring $A$ is a hyperfield if each non-zero element in $A$ is unit. $\varnothing \neq  B \subseteq A$ refers to  a  hyperideal  if {\bf 1.}  $x - y \in B$ for any $x, y \in B$,  {\bf 2.} $r \circ x \subseteq B$ for any $x \in B$ and $r \in A$ \cite{f10}.
 A proper hyperideal $B$ of   $A$ is a prime hyperideal if whenever $x,y \in A$ and $x \circ y \subseteq B$, then $x \in B$ or $y \in B$ \cite{das}.
  For any given hyperideal $B$ of $A$, the prime radical of $B$,  denoted by $rad (B)$, is the intersection of all prime hyperideals of $A$ containing  $B$. If the multiplicative hyperring $A$  has no prime hyperideal containing $B$, then we  define $rad(B)=A$ \cite{das}. We say that a proper hyperideal $B$ of   $A$ is a primary hyperideal if whenever $x,y \in A$ and $x \circ y \subseteq B$, then $x \in B$ or $y \in rad(B)$ \cite{das}. Let us write $a_1 \circ a_2 \circ \cdots \circ a_n$ as $\bigcirc_{i=1}^n a_i$. A hyperideal $B$ of $A$ is said to be a $\mathcal{C}$-hyperideal  if $\bigcirc_{i=1}^n a_i \cap B \neq \varnothing $ for $a_i \in A$ and $n \in \mathbb{N}$ imply $\bigcirc_{i=1}^n a_i \subseteq B$. Notice that in a multiplicative hyperring $A$, we have
 $\{a \in A \ \vert \  a^n \subseteq B \ \text{for some} \ n \in \mathbb{N}\} \subseteq rad(A)$. Proposition 3.2 in \cite{das} shows that the equality holds if $B$ is a $\mathcal{C}$-hyperideal of $A$. Furthermore, a hyperideal $B$ of $A$ is  a strong $\mathcal{C}$-hyperideal if $\sum_{i=1}^n(\bigcirc_{j=1}^k x_{ij}) \cap B \neq \varnothing$ for $x_{ij} \in A$ and  $k_i, n \in  \mathbb{N}$, then  $\sum_{i=1}^n(\bigcirc_{j=1}^k x_{ij}) \subseteq  B$. For more details you can see \cite{phd}. 
  A proper hyperideal $B$ in $A$ is maximal  if for
any hyperideal $M$ of $A$ with $B \subset M \subseteq A$, then $M = A$ \cite{ameri}. We denote the intersection of all maximal hyperideals of $A$  by $J(A)$. Also, the multiplicative hyperring $A$ is local  if it has just one maximal hyperideal \cite{ameri}.   For any give hyperideals  $B_1$ and $B_2$   of $A$, we define $(B_2:B_1)=\{a \in A \ \vert \ a \circ B_1 \subseteq B_2\}$ \cite{ameri}. A multiplicative hyperring  $A$ with identity $e$ is  hyperdomain if $0 \in x \circ y$ for  $x,y \in A$, then $x=0$ or $y=0$.

The concepts of prime and primary hyperideals have been extended to more nuanced concepts (\cite{Ay}, \cite{far}, \cite{Ghiasvand}, \cite{Ghiasvand2}, \cite{Mena}, \cite{Sen}, \cite{ul}). These expansions have paved the way for further discovery into more complex types of hyperideals. Despite these developments, there is a need to explore other classes of hyperideals for deeper understanding of commutative multiplicatve hyperrings structure.

In this paper, we aim to introduce the notion $(u,v)$-absorbing primary hyperideals  where $u,v \in \mathbb{Z}$ with $u>v$. Indeed,  this definition presents a distinct viewpoint by including the radical of  hyperideals   in the definition of  $(u,v)$-absorbing prime hyperideals proposed in \cite{anb8}.
 Among many results in this paper, we give an example of a $(u,v)$-absorbing primary hyperideal that is not a $(u,v)$-absorbing prime hyperideal (Example \ref{er1}). Although every $(u,v)$-absorbing primary  hyperideal of $A$ is a $(w,v)$-absorbing primary hyperideal of $A$ for any $w \geq u$, Example \ref{er2} shows that the converse may fail. In Theorem \ref{1}, we conclude that the radical of a $(u,v)$-absorbing primary $\mathcal{C}$-hyperideal is a prime  hyperideal. We obtain that in a local multiplicative hyperring $A$ with the maximal hyperideal $M$, $P \circ M$ is a $(u,v)$-absorbing primary hyperideal of $A$  where $P$ is a prime $\mathcal{C}$-hyperideal of $A$ in Proposition \ref{2}. In Theorem \ref{6}, we show that if there exists a $(u+1,v+1)$-absorbing primary strong $\mathcal{C}$-hyperideal of $A$ or a $(u+1,v)$-absorbing primary strong $\mathcal{C}$-hyperideal of $A$ that is not a $(u,v)$-absorbing primary hyperideal, then $A$ is a local multiplicative hyperring. Theorem \ref{7} presents a case where $(u,v)$-absorbing primary hyperideals are primary hyperideals. In Proposition \ref{ignor}, it is shown that if $P_1,\ldots,P_n$ are $(u,v)$-absorbing primary $\mathcal{C}$-hyperideals of $A$ such that $rad(P_i)=rad(P_j)$ for all $i,j \in \{1,\ldots,n\}$, then the intersection of the $P_i^,$s is a $(u,v)$-absorbing primary hyperideal. We give Example \ref{ignor1}  to show the condition  ``$rad(P_i)=rad(P_j)$ for all $i,j \in \{1,\ldots,n\}$" in Proposition \ref{ignor} is crucial. Moreover, we investigate the stability of $(u,v)$-absorbing primary hyperideals in various hyperring-theoretic constructions. 

Throughout this study, $A$  denotes a commutative multiplicative hyperring with identity element $1$.
\section{  $(u,v)$-absorbing primary hyperideals }
 A proper hyperideal of $P$ of $A$ is said to be 1-absorbing primary hyperideal if  whenever $x,y,z \in A\backslash U(A)$ and $x \circ y \circ z \subseteq P$, then $x \circ y \subseteq  P$ or  $z \in rad(P)$ \cite{anb1}. Now, we aim to generalize this concept to notion of $(u,v)$-absorbing primary  hyperideals and  give some
fundamental theorems and examples about them. We begin with the deﬁnition.

\begin{definition}
Let $P$ be a proper hyperideal of $A$ and $u,v \in \mathbb{N}$ with $u >v$. $P$ refers to  a   $(u,v)$-absorbing primary  hyperideal  if     $x_1 \circ \cdots \circ x_u \subseteq P$ for $x_1,\ldots, x_u \in A \backslash U(A)$, then either $x_1 \circ \cdots \circ x_v \subseteq P$ or $x_{v+1} \circ \cdots \circ x_u \subseteq rad(P)$.
\end{definition}
\begin{example}
\begin{itemize}
\item[\rm{(i)}]~
Assume that $(\mathbb{Z}[i],+,\cdot)$ is the Gaussian integers ring. Consider the multiplicative hyperring $(A_{\Phi},+,\circ)$ where $A_{\phi}=\mathbb{Z}[i]$, $\Phi=\{-1,3\}$ and for any $a,b \in A_{\Phi}$, $a \circ b=\{a \cdot x \cdot b \ \vert \ x \in \Phi\}$. In the hyperring, $P=2\mathbb{Z}[i]=\{-2x-2yi,6x+6yi \ \vert \ x,y \in \mathbb{Z}\}$ is a $(u,v)$-absorbing primary hyperideal of $A_{\Phi}$ for all $u,v \in \mathbb{N}$ with $u >v$.

\item[\rm{(ii)}]~ 
Consider the ring of polynomials $\mathbb{Z}[x]$ where $(\mathbb{Z},+,\cdot)$ is the ring of integers. Let  $A=\mathbb{Z}+3x\mathbb{Z}[x]$, $\alpha \circ \beta=\{2\alpha \beta,4\alpha  \beta\}$ for each $\alpha, \beta\in \mathbb{Z}$ and $P=3x\mathbb{Z}[x]$.  In the hyperring,  $P^2$ is a $(3,2)$-absorbing primary hyperideal of $A$.
\end{itemize}
\end{example}
Recall from \cite{anb8} that a proper hyperideal $P$ of $A$ is  a   $(u,v)$-absorbing prime  hyperideal  if     $x_1 \circ \cdots \circ x_u \subseteq P$ for $x_1,\ldots, x_u \in A \backslash U(A)$ implies either $x_1 \circ \cdots \circ x_v \subseteq P$ or $x_{v+1} \circ \cdots \circ x_u \subseteq P$. 
\begin{remark} \label{r1}
Every $(u,v)$-absorbing prime  hyperideal of $A$ is a $(u,v)$-absorbing primary hyperideal of $A$. 
\end{remark}
The following example shows that the converse of Remark \ref{r1} may not be
true, in general.
\begin{example} \label{er1}
Consider the multiplicative hyperring $(\mathbb{Z}_{\Phi},+,\circ)$ where $\Phi=\{2,3\}$. In the hyperring, $P=\langle 12 \rangle$ is a $(4,2)$-absorbing primary hyperideal of $A$. However, it is not a $(4,2)$-absorbing prime hyperideal of $A$ since $2^3 \circ 3=\{192,288,432,648\} \subseteq P$ but $2 \circ 2=\{8,12\} \nsubseteq P$ and $2 \circ 3=\{12,18\} \nsubseteq P$.
\end{example}
\begin{remark} \label{r2}
\begin{itemize}
\item[\rm{(i)}]~ Every $(u,v)$-absorbing primary  hyperideal of $A$ is a $(u+1,v+1)$-absorbing primary hyperideal of $A$. 

\item[\rm{(ii)}]~ Every $(u,v)$-absorbing primary  hyperideal of $A$ is a $(w,v)$-absorbing primary hyperideal of $A$ for any $w \geq u$. 
\end{itemize}
\end{remark}
The next example is given to explain that the converse of \ref{r2} (ii) may not be always true.
\begin{example} \label{er2}
In Example \ref{er1}, the $(4,2)$-absorbing prime hyperideal $P$ is not $(3,2)$-absorbing primary since $2 \circ 2 \circ 3=\{48,72,108\} \subseteq P$ while $2 \circ 2=\{8,12\} \notin P$, $2 \circ 3=\{12,18\} \nsubseteq P$ and $2,3 \notin rad(P)$.
\end{example}
The following theorem shows that the radical of a $(u,v)$-absorbing primary $\mathcal{C}$-hyperideal of $A$ is a prime hyperideal of $A$.
\begin{theorem} \label{1}
Let $u,v \in \mathbb{N}$ with $u >v$. If $P$ is a $(u,v)$-absorbing primary $\mathcal{C}$-hyperideal of $A$, then $rad(P)$ is a prime hyperideal of $A$.
\end{theorem}
\begin{proof}
Let $P$ be a $(u,v)$-absorbing primary hyperideal of $A$ and $a \circ b \subseteq rad(P)$ for $a,b \in A$. Let us assume that $a,b \in A\backslash U(A)$. Then we get $a^n \circ b^n \subseteq P$ for some $n \in \mathbb{N}$. It follows that  $a^{v-1} \circ a^n \circ b^{u-v-1} \circ b^n \subseteq P$. Take any $x \in a^n$ and $y \in b^n$. Let $x, y \in U(A)$. Then there exist $x^{\prime}, y^{\prime} \in A$ such that $1 \in x \circ x^{\prime}$ and $1 \in y \circ y^{\prime}$. Hence $1 \in a^n \circ x^{\prime}=a \circ (a^{n-1} \circ x^{\prime})$ and $1 \in b^n \circ x^{\prime}=b \circ (b^{n-1} \circ y^{\prime})$. Therefore we  have $1 \in a \circ p$ for some $p \in a^{n-1} \circ x^{\prime}$ and $1 \in b \circ q$ for some $q \in b^{n-1} \circ y^{\prime}$. This means that $a,b \in U(A)$, a contradiction.  Thus $x,y \notin U(A)$. Now, put $x_1=\cdots=x_{v-1}=a$, $x_v=x$, $x_{v+1}=\cdots=x_{u-1}=b$ and $x_u=y$. This implies that $a^{v-1} \circ x=x_1 \circ \cdots \circ x_{v-1} \circ x_v \subseteq P$ or $b^{u-v-1} \circ y=x_{v+1} \circ \cdots \circ x_{u-1} \circ x_u \subseteq rad(P)$ as $P$ is a $(u,v)$-absorbing primary hyperideal of $A$. Since $P$ is a $\mathcal{C}$-hyperideal of $A$,  we conclude that $a^{v+n-1} \subseteq P$ which means $a \in rad(P)$ or $b^{u-v+n-1} \subseteq rad(P)$ which implies $b \in rad(P)$. Consequently, $rad(P)$ is a prime hyperideal of $A$.
\end{proof}
In the following proposition, we analyze when a product of two hyperideals is a $(u,v)$-absorbing primary hyperideal.
\begin{proposition} \label{2}
Let every hyperideal of $A$ be $\mathcal{C}$-hyperideal and $u,v \in \mathbb{N}$ with $u >v$. If $M$ is the only maximal hyperideal of $A$ and $P$ is a prime hyperideal of $A$, then $P \circ M $ is a $(u,v)$-absorbing primary hyperideal of $A$. 
\end{proposition}
\begin{proof}
Assume that $M$ is the only maximal hyperideal of $A$ and $P$ is  a prime $\mathcal{C}$-hyperideal of $A$. Assume that $x_1 \circ \cdots \circ x_u \subseteq P \circ M$ for $x_1,\ldots,x_u \in A \backslash U(A)$. Since $M$ is the only maximal hyperideal of $A$, $x_1,\ldots,x_u \in M$. Let $x_i \in P$ for some $i \in \{1,\ldots,v\}$. Then we get $x_1 \circ \cdots \circ x_v \subseteq P \circ M$. Now, we assume that $x_i \notin P$ for all $i \in \{1,\ldots,v\}$. This implies that $x_1 \circ \cdots \circ x_v \nsubseteq P \circ M$. Thus, we conclude that $x_{v+1} \circ \cdots \circ x_u \subseteq P =rad(P\circ M)$ as  $x_1 \circ \cdots \circ x_u \subseteq P \circ M \subseteq P$ and $x_1 \circ \cdots \circ x_v \nsubseteq P$. Hence, $P \circ M $ is a $(u,v)$-absorbing primary hyperideal of $A$. 
\end{proof}
Next, we investigate when $(P:x)$ is a $(u,v)$-absorbing primary hyperideal of $A$. 
\begin{proposition} \label{3}
Let $P$ be a $(u,v)$-absorbing primary hyperideal of $A$ where $u,v \in \mathbb{N}$ with $u >v$ and $x \in A \backslash (P \cup U(A))$. Then $(P : x)$ is a $(u-1,v-1)$-absorbing primary hyperideal of $A$. 
\end{proposition}
\begin{proof}
Let $x_1 \circ \cdots \circ x_{u-1} \subseteq (P : x)$ for $x_1, \ldots, x_{u-1} \in A \backslash U(A)$ and $x \in A \backslash (P \cup U(A))$ such that $x_1 \circ \cdots \circ x_{v-1} \nsubseteq (P : x)$. Therefore we get $x \circ x_1 \circ \cdots \circ x_{u-1} \subseteq P$. Since $P$ is a $(u,v)$-absorbing primary hyperideal of $A$ and $x \circ x_1 \circ \cdots \circ x_{v-1} \nsubseteq P$, we conclude that $x_v \circ \cdots x_{u-1} \ \subseteq rad(P) \subseteq rad((P : x))$. This shows that $(P : x)$ is a $(u-1,v-1)$-absorbing primary hyperideal of $A$. 
\end{proof}
The following lemma is needed in the proof of our next result.
\begin{lemma} \label{4}
Let $P$ be a strong $\mathcal{C}$-hyperideal of $A$. Then so is $rad(P)$. 
\end{lemma}
\begin{proof}
Assume that $P$ is a strong $\mathcal{C}$-hyperideal of $A$. Let $\sum_{i=1}^n(\bigcirc_{j=1}^k x_{ij}) \cap rad(P) \neq \varnothing$ for $x_{ij} \in A$ and  $k_i, n \in  \mathbb{N}$. This means that there exists $x \in \sum_{i=1}^n(\bigcirc_{j=1}^k x_{ij}) \cap rad(P)$. Then we conclude  that $x^t \subseteq P$ for some $t \in \mathbb{N}$. Since $P$ is a strong $\mathcal{C}$-hyperideal of $A$ and $(\sum_{i=1}^n(\bigcirc_{j=1}^k x_{ij}))^t \cap P \neq \varnothing$, we get $(\sum_{i=1}^n(\bigcirc_{j=1}^k x_{ij}))^t \subseteq P$ and so $\sum_{i=1}^n(\bigcirc_{j=1}^k x_{ij}) \subseteq rad(P)$, as needed.
\end{proof}
Let $a \in P$ where $P$ is a strong $\mathcal{C}$-hyperideal of $A$. We next handle the case when $a+1$ is a nonunit of $A$.
\begin{theorem} \label{5}
Let $P$ be a strong $\mathcal{C}$-hyperideal of $A$ such that $a+1 \in A \backslash U(A)$ for some $a \in P$ and  $u,v \in \mathbb{N}$ with $u >v$. Then $P$ is a $(u,v)$-absorbing primary  hyperideal of $A$ if and only if  $x_1 \circ \cdots \circ x_u \subseteq P$ for $x_1,\ldots, x_u \in A $ implies that $x_1 \circ \cdots \circ x_v \subseteq P$ or $x_{v+1} \circ \cdots \circ x_u \subseteq rad(P)$.
\end{theorem}
\begin{proof}
$\Longrightarrow$ Let $P$ is a $(u,v)$-absorbing primary  hyperideal of $A$ and $x_1 \circ \cdots \circ x_u \subseteq P$ for $x_1,\ldots, x_u \in A $. Put $t=\min \{i \ \vert \ 1 \leq i \leq u , \ x_i \notin U(A)\}$. Clearly, we have $x_{v+1} \circ \cdots \circ x_u \subseteq P \subseteq rad(P)$ if $t \geq v+1$. Assume that $t \leq v$. Put $X=\{i \ \vert \   1 \leq i \leq v, x_i \notin U(A)\}$ and $Y=\{i \ \vert \   v+1 \leq i \leq u, x_i \notin U(A)\}$. We suppose that $card(X)=m$ and $card(Y)=n$. Therefore we get $(\bigcirc_{k \in X}x_k \circ (a+1)^{v-m}) \circ (\bigcirc_{k \in Y}x_k \circ (a+1)^{u-v-n}) \subseteq P$. Since $P$ is a $(u,v)$-absorbing primary  hyperideal of $A$, we conclude that $\bigcirc_{k \in X}x_k \circ (a+1)^{v-m} \subseteq P$ or $ \bigcirc_{k \in Y}x_k \circ (a+1)^{u-v-n} \subseteq rad(P)$. On the other hand, we have  $a \circ I +1 \subseteq (a+1)^{v-m}=\sum_{j=0}^{v-m} \tbinom{v-m}{j}a^{v-m-j} \circ 1^j$ for some  $I \subseteq A$ and $a \circ J +1 \subseteq (a+1)^{u-v-n}=\sum_{j=0}^{u-v-n} \tbinom{u-v-n}{j}a^{u-v-n-j} \circ 1^j$ for some  $J \subseteq A$. Then we obtain $\bigcirc_{k \in X}x_k \circ (a \circ I+1) \subseteq P$ or $\bigcirc_{k \in Y}x_k \circ (a \circ J +1) \subseteq rad(P)$. Since $rad(P)$ is a strong $\mathcal{C}$-hyperideal of $A$ by Lemma \ref{4} and $a  \in P$, we get $\bigcirc_{k \in X}x_k  \subseteq P$ or $\bigcirc_{k \in Y}x_k \subseteq rad(P)$ and so $x_1 \circ \cdots \circ x_v \subseteq P$ or $x_{v+1} \circ \cdots \circ x_u \subseteq rad(P)$. 

$\Longleftarrow$ It is obvious.
\end{proof}
\begin{theorem} \label{6}
Let $u,v \in \mathbb{N}$ with $u >v$. If there exists a $(u+1,v+1)$-absorbing primary strong $\mathcal{C}$-hyperideal of $A$ or a $(u+1,v)$-absorbing primary strong $\mathcal{C}$-hyperideal of $A$ that is not a $(u,v)$-absorbing primary hyperideal, then $A$ is a local multiplicative hyperring.
\end{theorem}
\begin{proof}
Assume that $P$ is a strong $\mathcal{C}$-hyperideal of $A$ that is not a $(u,v)$-absorbing primary hyperideal. This means that  $x_1\circ \cdots \circ x_u \subseteq P$ for some $x_1,\ldots,x_u \in A \backslash U(A)$  but $x_1\circ \cdots \circ x_v \nsubseteq P$ and $x_{v+1} \circ \cdots \circ x_u \nsubseteq rad(P)$. Let $P$ be  a $(u+1,v+1)$-absorbing primary strong $\mathcal{C}$-hyperideal of $A$. Take any $x \in A \backslash U(A)$. Since $x \circ x_1 \circ \cdots \circ x_u \subseteq P$ and $x_{v+1} \circ \cdots \circ x_u \nsubseteq rad(P)$, we conclude that $x \circ x_1\circ \cdots \circ x_v \subseteq P$. Take any $y \in U(A)$. We show that $x+y  \in U(A)$. Let $x+y  \notin U(A)$. From $(x+y) \circ x_1 \circ \cdots \circ x_u \subseteq P$, it follows that $(x+y) \circ x_1\circ \cdots \circ x_v \subseteq P$ because $x_{v+1} \circ \cdots \circ x_u \nsubseteq rad(P)$. Since $P$ is a strong $\mathcal{C}$-hyperideal and $(x \circ x_1\circ \cdots \circ x_v)+(y \circ x_1\circ \cdots \circ x_v) \cap P \neq \varnothing$, we obtain $(x \circ x_1\circ \cdots \circ x_v)+(y \circ x_1\circ \cdots \circ x_v) \subseteq  P$. Since $x \circ x_1 \circ \cdots \circ x_v \subseteq P$, we have $y \circ x_1\circ \cdots \circ x_v \subseteq  P$ which implies $x_1\circ \cdots \circ x_v \subseteq  P$ which is impossible. Hence $x+y  \in U(A)$ which implies $A$ is a local multiplicative hyperring by Lemma 2.6 in \cite{Ghiasvand2}. Now, suppose that $P$ is a $(u+1,v)$-absorbing primary strong $\mathcal{C}$-hyperideal of $A$, $x \in A \backslash U(A)$ and $y \in U(A)$. Again, we show that $x+y  \in U(A)$ and so we are done by Lemma 2.6 in \cite{Ghiasvand2}. Assume that $x+y  \notin U(A)$.  From $ x_1\circ \cdots \circ x_u \circ (x+y) \subseteq P$, it follows that $x_{v+1} \circ \cdots \circ x_u \circ (x+y)\subseteq rad(P)$ as $P$ is a $(u+1,v)$-absorbing primary hyperideal of $A$ and $x_1\circ \cdots \circ x_v \nsubseteq P$. Since $(x_{v+1} \circ \cdots \circ x_u \circ x)+(x_{v+1} \circ \cdots \circ x_u \circ y)\cap rad(P) \neq \varnothing $, we obtain  $(x_{v+1} \circ \cdots \circ x_u \circ x)+(x_{v+1} \circ \cdots \circ x_u \circ y)\subseteq  rad(P) $ by Lemma \ref{4}. This implies that $x_{v+1} \circ \cdots \circ x_u \subseteq  rad(P) $ as $x_{v+1} \circ \cdots \circ x_u \circ x \subseteq rad(P)$ and $y \in U(A)$. This is a contradiction and so $x+y  \in U(A)$.
\end{proof}
The next result gives a case where $(u,v)$-absorbing primary hyperideals are primary hyperideals. 
\begin{theorem} \label{7}
Let $P$ be a strong $\mathcal{C}$-hyperideal of $A$ such that $A$ is not a local multiplicative hyperring and $u,v \in \mathbb{N}$ with $u >v\geq2$. Then $P$ is a $(u,v)$-absorbing primary hyperideal of $A$ if and only if $P$ is a primary hyperideal of $A$.
\end{theorem}
\begin{proof}
$\Longrightarrow$ Assume that $A$ is not a local multiplicative hyperring and $P$ is a $(u,v)$-absorbing primary hyperideal of $A$. Therefore  $P$ is a $(u-1,v-1)$-absorbing primary hyperideal of $A$  by Theorem \ref{6}. Therefore we conclude that  $P$ is a $(u-v+1,1)$-absorbing primary hyperideal of $A$. Now, let $a \circ b \subseteq P$ for $a,b \in A$ but $a \notin P$. So $a \circ b^{u-v} \subseteq P$. Since $P$ is a $(u-v+1,1)$-absorbing primary hyperideal and $a \notin P$, we obtain $b^{u-v} \subseteq rad(P)$. Take any $x \in b^{u-v}$. Then $x^n \subseteq P$ for some $n \in \mathbb{N}$. Since $b^{n(u-v)} \cap P \neq \varnothing$, we have $b^{n(u-v)} \subseteq  P$ which means $b \in rad(P)$. Thus $P$ is a primary hyperideal of $A$.

$\Longleftarrow$ It is obvious.
\end{proof}
\begin{theorem}
Let $P$ be a strong $\mathcal{C}$-hyperideal of $A$  and $u,v \in \mathbb{N}$ with $u >v$. If $P$ is a $(u+1,v)$-absorbing primary hyperideal of $A$, then $P$ is a $(u,v)$-absorbing primary hyperideal of $A$ or $A$ is a local multiplicative hyperring with maximal hyperideal $M$ such that $rad(P)=M$.
\end{theorem}
\begin{proof}
Let  $P$ be a $(u+1,v)$-absorbing primary hyperideal of $A$. Assume that $P$ is not a $(u,v)$-absorbing primary hyperideal of $A$. Then there exist $x_1,\ldots,x_u \in A \backslash U(A)$ such that $x_1 \circ \cdots \circ x_u \subseteq P$, $x_1\circ \cdots \circ x_v \nsubseteq P$ and $x_{v+1}\circ \cdots \circ x_u \nsubseteq rad(P)$. On the other hand,   we conclude that $A$ is a local multiplicative hyperring by Theorem \ref{6} as $P$ is a $(u+1,v)$-absorbing primary hyperideal of $A$  that   is not a $(u,v)$-absorbing primary hyperideal of $A$. Let $M$ be the only maximal hyperideal of $A$ and $a \in M$. Since $P$ is a $(u+1,v)$-absorbing primary hyperideal of $A$, $x_1\circ \cdots \circ x_u \circ a \subseteq P$ and $x_1\circ \cdots \circ x_v \nsubseteq P$, we get $x_{v+1}\circ \cdots \circ x_u \circ a \subseteq rad(P)$. By Theorem \ref{1}, $rad(P)$ is a prime hyperideal of $A$. Let  $x \in x_{v+1}\circ \cdots \circ x_u $. If $x \in rad(P)$, then $x_{v+1}\circ \cdots \circ x_u  \subseteq rad(P)$, a contradiction. Then $x \notin rad(P)$. Since $x \circ a \subseteq rad(P)$, we get $a \in rad(P)$ which shows $rad(P)=M$, as required. 
\end{proof}
\begin{proposition} \label{ignor}
Let $P_1, \ldots,P_n$ be $(u,v)$-absorbing primary $\mathcal{C}$-hyperideals of $A$ where $u,v \in \mathbb{N}$ with $u >v$. If $rad(P_i)=rad(P_j)$ for all $i,j \in \{1,\ldots,n\}$, then $\cap_{i=1}^nP_i$ is $(u,v)$-absorbing primary hyperideal of $A$.
\end{proposition}
\begin{proof}
Assume that  $P_1, \ldots,P_n$ are $(u,v)$-absorbing primary $\mathcal{C}$-hyperideals of $A$. Then $rad(P_i)$ is a prime hyperideal of $A$ for all $i \in \{1,\ldots,n\}$ by Theorem \ref{1}. Let us assume $rad(P_i)=Q$ for all $i \in \{1,\ldots,n\}$ where $Q$ is a prime hyperideal of $A$. Put $P=\cap_{i=1}^nP_i$.  Let $x_1\circ \cdots \circ x_u \subseteq P$ for $x_1,\ldots,x_u \in A \backslash U(A)$ such that $x_1\circ \cdots \circ x_v \nsubseteq P$. Then there exists  $t \in \{1,\ldots,n\}$ such that $x_1\circ \cdots \circ x_v \nsubseteq P_t$. Since $P_t$ is a $(u,v)$-absorbing primary hyperideals of $A$ and $x_1\circ \cdots \circ x_u \subseteq P_t$, we conclude that $x_{v+1}\circ \cdots \circ x_u \subseteq rad(P_t)=Q=rad(P)$ by Proposition 3.3 in \cite{das}. Thus $P=\cap_{i=1}^nP_i$ is a $(u,v)$-absorbing primary hyperideal of $A$.
\end{proof}
Note that the conditions ``$rad(P_i)=rad(P_j)$ for all $i,j \in \{1,\ldots,n\}$"in Proposition \ref{ignor} can not be ignored.

\begin{example} \label{ignor1}
Consider the multiplicative hyperring $(\mathbb{Z}_{\Phi},+,\circ)$ where $\Phi=\{2,4\}$. The hyperideals $\langle 3 \rangle$, $\langle 5 \rangle$ and $\langle 7 \rangle$ are $(3,2)$-absorbing primary $\mathcal{C}$-hyperideals of $A$ but  $rad(\langle 3 \rangle) \neq rad(\langle 5 \rangle) \neq rad(\langle 7 \rangle)$ and $\langle 150  \rangle=\langle 3 \rangle \cap \langle 5 \rangle \cap \langle 7 \rangle$ is not $(3,2)$-absorbing primary hyperideals of $A$.
\end{example}
Now, we give a a characterization of $(v+1,v)$-absorbing primary hyperideals of $A$. 
\begin{theorem}
Let $P$ be a $\mathcal{C}$-hyperideal of $A$ and $u,v \in \mathbb{N}$ with $u >v$. Then the followings are equivalent.
\begin{itemize}
\item[\rm{(i)}]~ $P$ is a $(v+1,v)$-absorbing primary hyperideal of $A$. 
\item[\rm{(ii)}]~ If $a_1 \circ \cdots \circ a_v \nsubseteq P$ for all $a_1, \ldots,a_v \in A \backslash U(A)$, then $(P : a_1 \circ \cdots \circ a_v) \subseteq rad(P)$.
\item[\rm{(iii)}]~ If $a_1 \circ \cdots \circ a_v \circ Q \subseteq P$ for some hyperideal $Q$ of $A$ and  $a_1, \ldots,a_v \in A \backslash U(A)$, then $a_1 \circ \cdots \circ a_v \subseteq P$ or $Q \subseteq rad(P)$.
\item[\rm{(iv)}]~If $P_1 \circ \cdots \circ P_v \circ P_{v+1} \subseteq P$ for proper hyperideals  $P_1, \ldots,P_{v+1}$ of $A$, then $P_1 \circ \cdots \circ P_v \subseteq P$ or $P_{v+1} \subseteq rad(P)$.
\end{itemize}
\end{theorem}
\begin{proof}
(i) $\Longrightarrow$ (ii) Assume that $P$ is a $(v+1,v)$-absorbing primary hyperideal of $A$ and $a_1 \circ \cdots \circ a_v \nsubseteq P$ for all $a_1, \ldots,a_v \in A \backslash U(A)$. Let $b \in (P : a_1 \circ \cdots \circ a_v)$.  So $a_1 \circ \cdots \circ a_v \circ b \subseteq P$. Assume that $b \in U(A)$. Then we have $a_1 \circ \cdots \circ a_v \subseteq a_1 \circ \cdots \circ a_v \circ 1 \subseteq a_1 \circ \cdots \circ a_v \circ b \circ b^{-1} \subseteq P$, a contradiction. Then $b \notin U(A)$. By the hypothesis, we conclude that $b \in rad(P)$ as $a_1 \circ \cdots \circ a_v \nsubseteq P$. Thus we have $(P : a_1 \circ \cdots \circ a_v) \subseteq rad(P)$.

(ii) $\Longrightarrow$ (iii) Assume that $a_1 \circ \cdots \circ a_v \circ Q \subseteq P$ for some hyperideal $Q$ of $A$ and  $a_1, \ldots,a_v \in A \backslash U(A)$  such that $a_1 \circ \cdots \circ a_v \nsubseteq P$. From $a_1 \circ \cdots \circ a_v \circ Q \subseteq P$, it follows that $Q \subseteq (P : a_1 \circ \cdots \circ a_v)$. Thus,  we conclude that $Q \subseteq rad(P)$ by (ii). 

(iii) $\Longrightarrow$ (iv) Let $P_1 \circ \cdots \circ P_v \circ P_{v+1} \subseteq P$ for proper hyperideals  $P_1, \ldots,P_v, P_{v+1}$ of $A$ such that $P_1 \circ \cdots \circ P_v \nsubseteq P$. This implies that $a_1\circ \cdots \circ a_v \nsubseteq P$ for some $a_1 \in P_1, \ldots, a_v \in P_v$. From $a_1\circ \cdots \circ a_v \circ P_{v+1} \subseteq P$ it follows that $P_{v+1} \subseteq rad(P)$ by (iii). 

(iv) $\Longrightarrow$ (i) Let $a_1 \circ \cdots \circ a_v \circ a_{v+1} \subseteq P$ for $a_1,\dots, a_v, a_{v+1} \in A \backslash U(A)$. Then we have  $\langle a_1 \rangle \circ \cdots \circ \langle a_v \rangle \circ \langle a_{v+1} \rangle \subseteq \langle a_1 \circ \cdots \circ a_v \circ a_{v+1} \rangle \subseteq P$ by Proposition 2.15 in \cite{das}. By (iv), we conclude that $\langle a_1 \rangle \circ \cdots \circ \langle a_v \rangle \subseteq P$ or $\langle a_{v+1} \rangle \subseteq  rad(P)$. This implies that $a_1 \circ \cdots \circ a_v \subseteq P$ or $a_{v+1} \in rad(P)$. Consequently, $P$ is a $(v+1,v)$-absorbing primary hyperideal of $A$. 
\end{proof}

Assume that $A$ is  a hyperdomain with quotient hyperﬁeld $F$. Recall from \cite{Ghiasvand2} that a hyperdomain
 $A$ is said to be a Dedekind hyperdomain if every nonzero proper hyperideal of $A$ is invertible, that is, for every nonzero proper hyperideal $P$ of $A$, $P\circ P^{-1}=A$ where $P^{-1}=\{x \in F \ \vert \ x \circ P \subseteq  A\}$. Now, we examine  when a $\mathcal{C}$-hyperideal in  a Dedekind hyperdomain is  a $(u,v)$-absorbing primary hyperideal.
\begin{theorem}
Let $A$ be a Dedekind hyperdomain, $P$ a $\mathcal{C}$-hyperideal of $A$ and  $u,v \in \mathbb{N}$ with $u >v$. Then $P$ is a $(u,v)$-absorbing primary hyperideal of $A$ if and only if $rad(P)$ is a prime hyperideal of $A$.
\end{theorem}
\begin{proof}
$\Longrightarrow$ Since $P$ is a $(u,v)$-absorbing primary hyperideal of $A$, the calim follows from \ref{1}.

$\Longleftarrow$ Let $rad(P)$ be a prime hyperideal of $A$. By the hypothesis, $rad(P)$ is maximal. Assume that $x \circ y \subseteq P$ and $y \notin rad(P)$. Since $rad(P)$ is maximal hyperideal of $A$, then $\langle y, rad(P) \rangle=A$. Then there exists $a \in rad(P)$ and $s \in A$ such that $1 \in y \circ s +a$. So, we obtain $1=r+a$ for some  $r \in y \circ s$. Since $a \in rad(P)$, there exists $n \in \mathbb{N}$ such that $a^n \subseteq P$. Therefore we have $1 \in (r+a)^n \subseteq \sum_{j=0}^{n-1} \tbinom{n}{j}r^{n-j}\circ a^j+  a^n$.  If follows that $x \in x \circ 1 \subseteq x \circ (\sum_{j=0}^{n-1} \tbinom{n}{j}r^{n-j}\circ a^j+a^n )\subseteq \sum_{j=0}^{n-1} \tbinom{n}{j} x \circ r^{n-j}\circ a^j+x \circ a^n \subseteq \sum_{j=0}^{n-1} \tbinom{n}{j} x \circ (y\circ s)^{n-j}\circ a^j+x \circ a^n \subseteq P$ which shows $P$ is a primary hyperideal of $A$. Thus we conclude that $P$ is a $(u,v)$-absorbing primary hyperideal of $A$.
\end{proof}
Recall from \cite{Ghiasvand2} that a hyperring is divided if for each prime  hyperideal $Q$ of $A$, we have $Q \subseteq \langle a \rangle$ for every $a \in A\backslash Q$. The following theorem shows that a $\mathcal{C}$-hyperideal of a divided multiplicative hyperring $A$ is a $(u,v)$-absorbing primary hyperideal  if and only if it is a primary hyperideal.
\begin{theorem}
Let $A$ be a divided multiplicative hyperring, $P$ a $\mathcal{C}$-hyperideal of $A$ and  $u,v \in \mathbb{N}$ with $u >v$. Then $P$ is a $(u,v)$-absorbing primary hyperideal of $A$ if and only if $P$ is a primary hyperideal of $A$.
\end{theorem}
\begin{proof}
$\Longrightarrow $ Let $P$ be a $(u,v)$-absorbing primary hyperideal of $A$ and $a \circ b \subseteq P$ for $a,b \in A \backslash U(A)$ but $ b\notin rad(P)$. Since $P$ is a $(u,v)$-absorbing primary hyperideal of $A$, $rad(P)$ is a prime hyperideal of $A$ by Theorem \ref{1}. This implies that $a \in rad(P)$ as $a \circ b \subseteq rad(P)$ and $b \notin rad(P)$. Take any $c \in b^{v-1}$. Since $A$ is  a divided multiplicative hyperring and $c \notin rad(P)$, we get $rad(P) \subseteq \langle c \rangle$. Since $a \in rad(P)$, there exists $r \in A$ such that $a \in c \circ r$. Let $r \in U(A)$. Then $a \circ r^{-1} \subseteq  c \circ r \circ r^{-1}$. Since $rad(P)$ is a $\mathcal{C}$-hyperideal of $A$ and $ c \circ r \circ r^{-1} \cap rad(P) \neq \varnothing$, we obtain $c \in c \circ 1 \subseteq c \circ r \circ r^{-1} \subseteq rad(P)$, a contradiction. Hence Let $r \notin U(A)$. Since $P$ is a $(u,v)$-absorbing primary hyperideal of $A$ and $a \circ b ^{u-v} \subseteq b^{v-1} \circ r \circ b^{u-v} \subseteq P$, we conclude that $b^{v-1} \circ r \subseteq P$ and so $a \in P$. Consequently, $P$ is a primary hyperideal of $A$.

$\Longleftarrow$ It is straightforward. 
\end{proof}
Recall from \cite{f10} that a mapping $\eta$ from the commutative multiplicative hyperring
$(A_1, +_1, \circ _1)$ into the commutative multiplicative hyperring $(A_2, +_2, \circ _2)$ is a hyperring good homomorphism if $\eta(x +_1 y) =\eta(x)+_2 \eta(y)$ and $\eta(x\circ_1y) = \eta(x)\circ_2 \eta(y)$ for every $x,y \in A_1$.

\begin{theorem} \label{homo} 
Suppose that $A_1$ and $A_2$ are two commutative multiplicative hyperrings such that the mapping $\eta$ from $ A_1$ into $ A_2$ is a hyperring
good homomorphism,  $\eta(x) \notin U(A_2)$ for each $x \in A_1 \backslash U(A_1)$ and $u,v \in \mathbb{N}$ with $u >v$. Then the followings are satisfied:
\begin{itemize}
\item[\rm{(i)}]~ If $P_2$ is a $(u,v)$-absorbing primary $\mathcal{C}$-hyperideal of $A_2$, then $\eta^{-1}(P_2)$ is a $(u,v)$-absorbing primary  $\mathcal{C}$-hyperideal of $A_1$.
\item[\rm{(ii)}]~ If  $P_1$ is a is a $(u,v)$-absorbing primary $\mathcal{C}$-hyperideal of $A_1$ with $Ker (\eta) \subseteq P_1$ and $\eta$ is surjective, then $\eta(P_1)$ is a $(u,v)$-absorbing primary  $\mathcal{C}$-hyperideal of $A_2$.
\end{itemize}
\end{theorem}
\begin{proof}
(i) Assume that $P_2$ is a $(u,v)$-absorbing primary $\mathcal{C}$-hyperideal of $A_2$. Since $P_2$ is a $\mathcal{C}$-hyperideal of $A_2$, we conclude that $\eta^{-1}(P_2)$ is a $\mathcal{C}$-hyperideal of $A_2$ by Proposisition 2.8(ii) in \cite{Sen}. Now, let $ x_1 \circ_1 \cdots \circ_1 x_u \subseteq \eta^{-1}(P_2)$ for $x_1,\ldots,x_u \in A_1 \backslash U(A_1)$. Then we have $ \eta(x_1 \circ_1 \cdots \circ_1 x_u)=\eta(x_1) \circ_2 \cdots \circ_2 \eta(x_u) \subseteq P_2$ as $\eta$ is a good homomorphism. Since $P_2$ is a $(u,v)$-absorbing primary hyperideal of $A_2$ and $\eta(x_1),\ldots,\eta(a_u) \notin U(A_2)$, we get $\eta(x_1 \circ_1 \cdots \circ_1 x_v)=\eta(x_1) \circ_2 \cdots \circ_2 \eta(x_v) \subseteq P_2$ which means $x_1 \circ_1 \cdots \circ_1 x_v \subseteq \eta^{-1}(P_2)$ or $\eta(x_{v+1} \circ_1 \cdots \circ_1 x_u)=\eta(x_{v+1}) \circ_2 \cdots \circ_2 \eta(x_u) \subseteq rad(P_2)$ which implies $x_{v+1} \circ_1 \cdots \circ_1 x_u \subseteq \eta^{-1}(rad(P_2)) \subseteq rad(\eta^{-1}(P_2))$. Hence,  $\eta^{-1}(P_2)$ is a $(u,v)$-absorbing primary hyperideal of $A_1$.

(ii) Suppose that $P_1$ is  a $(u,v)$-absorbing primary $\mathcal{C}$-hyperideal of $A_1$ such that $Ker (\eta) \subseteq P_1$ and $\eta$ is surjective. Since $P_1$ is a $\mathcal{C}$-hyperideal of $A_1$, $\eta(P_1)$ is a $\mathcal{C}$-hyperideal of $A_1$ by Proposisition 2.8(i) in \cite{Sen}. Now, assume that $ y_1 \circ_2 \cdots \circ_2 y_u \subseteq \eta(P_1)$ for $y_1, \ldots, y_u \in A_2 \backslash U(A_2)$. Then there exists $x_i \in A_1 \backslash U(A_1)$ with $\eta(x_i)=y_i$ for every $i \in \{1,\cdots,u\}$ as $\eta$ is surjective. Hence, we have $\eta(x_1 \circ_1 \cdots \circ_1 x_u)=\eta(x_1) \circ_2 \cdots \circ_2 \eta(x_u)\subseteq \eta(P_1)$. Now, take any $p \in x_1 \circ_1 \cdots \circ_1 x_u$. Then we get $\eta(p) \in \eta(x_1 \circ_1 \cdots \circ_1 x_u) \subseteq \eta(P_1)$ and so there exists $q \in P_1$ such that $\eta(p)=\eta(q)$. Then we get $\eta(p-q)=0$ which means $p-q \in Ker (\eta)\subseteq P_1$ and so $p \in P_1$. Therefore we conclude that $x_1 \circ_1 \cdots \circ_1 x_u \subseteq P_1$ as $P_1$ is a $\mathcal{C}$-hyperideal. Since $P_1$ is a $(u,v)$-absorbing primary hyperideal of $A_1$, we obtain $x_1 \circ_1 \cdots \circ_1 x_v \subseteq P_1$ or $x_{v+1} \circ_1 \cdots \circ_1 x_u \subseteq rad(P_1)$. This implies that $y_1 \circ_2 \cdots \circ_2 y_v =\eta(x_1 \circ_1 \cdots \circ_1 x_v )\subseteq \eta(P_1)$ or $y_{v+1} \circ_2 \cdots \circ_2 y_u=\eta(x_{v+1} \circ_1 \cdots \circ_1 x_u )\subseteq \eta(rad(P_1))=rad(\eta(P_1))$. Thus, $\eta(P_1)$ is a $(u,v)$-absorbing primary hyperideal of $A_2$.
\end{proof}
Now, we have the following result.
\begin{corollary}
Let the hyperideal $P$ of $A$ be a subset of the $\mathcal{C}$-hyperideal $Q$ of $A$, $x+P \notin U(A/P)$ for all $x \in A \backslash U(A)$ and $u,v \in \mathbb{N}$ with $u>v$. Then $Q$ is a $(u,v)$-absorbing primary hyperideal of $A$ if and only if $Q/P$ is a $(u,v)$-absorbing primary hyperideal of $A/P$.
\end{corollary}
\begin{proof}
Consider the good epimorphism $\eta :A \longrightarrow A/P$ defined by $\eta(a)=a+P$. Now,  the claim follows from Theorem \ref{homo}. 
\end{proof}
For any given multiplicative hyperring $A$,  $M_m(A)$ denotes the set of all hypermatrices of $A$. Let $I = (I_{ij})_{m \times m}, J = (J_{ij})_{m \times m} \in P^\star (M_m(A))$. Then $I \subseteq J$ if and only if $I_{ij} \subseteq J_{ij}$\cite{ameri}. 
\begin{theorem} \label{8} 
Let $u,v \in \mathbb{N}$ with $u >v$ and $P$ be a hyperideal of $A$. If $M_m(P)$ is a $(u,v)$-absorbing primary $\mathcal{C}$-hyperideal of $M_m(A)$, then $P$ is a $(u,v)$-absorbing primary $\mathcal{C}$-hyperideal of $A$. 
\end{theorem}
\begin{proof}
Let $x_1 \circ \cdots \circ x_n \cap P \neq \varnothing$ for $x_1,\ldots,x_n \in A$. Then we have 
\[\begin{pmatrix}
x_1 \circ \cdots \circ x_n & 0 & \cdots & 0 \\
0 & 0 & \cdots & 0 \\
\vdots & \vdots & \ddots \vdots \\
0 & 0 & \cdots & 0 
\end{pmatrix}
\cap M_m(P) \neq \varnothing\]
which means 
\[(\begin{pmatrix}
x_1 & 0 & \cdots & 0\\
0 & 0 & \cdots & 0\\
\vdots & \vdots & \ddots \vdots\\
0 & 0 & \cdots & 0
\end{pmatrix}
\circ \cdots \circ
\begin{pmatrix}
x_n & 0 & \cdots & 0\\
0 & 0 & \cdots & 0\\
\vdots & \vdots & \ddots \vdots\\
0 & 0 & \cdots & 0
\end{pmatrix}) \cap M_m(P) \neq \varnothing.
\]
Since $M_m(P)$ is a $\mathcal{C}$-hyperideal of $M_m(A)$, we obtain \[\begin{pmatrix}
x_1 & 0 & \cdots & 0\\
0 & 0 & \cdots & 0\\
\vdots & \vdots & \ddots \vdots\\
0 & 0 & \cdots & 0
\end{pmatrix}
\circ \cdots \circ
\begin{pmatrix}
x_n & 0 & \cdots & 0\\
0 & 0 & \cdots & 0\\
\vdots & \vdots & \ddots \vdots\\
0 & 0 & \cdots & 0
\end{pmatrix} \subseteq M_m(P).
\]
and so 
\[\begin{pmatrix}
x_1 \circ \cdots \circ x_n & 0 & \cdots & 0 \\
0 & 0 & \cdots & 0 \\
\vdots & \vdots & \ddots \vdots \\
0 & 0 & \cdots & 0 
\end{pmatrix}
\subseteq  M_m(P).\]
This means $x_1 \circ \cdots \circ x_n \subseteq P$. Hence $P$ is a $\mathcal{C}$-hyperideal of $A$.
Now, let $x_1 \circ \cdots \circ x_u \subseteq P$ for $x_1, \ldots , x_u \in A \backslash U(A)$. Then we obtain
\[\begin{pmatrix}
x_1 \circ \cdots \circ x_u & 0 & \cdots & 0 \\
0 & 0 & \cdots & 0 \\
\vdots & \vdots & \ddots \vdots \\
0 & 0 & \cdots & 0 
\end{pmatrix}
\subseteq M_m(P).\]
Since $M_m(P)$ is a $(u,v)$-absorbing primary hyperideal of $M_m(A)$ and 
\[ \begin{pmatrix}
x_1 \circ \cdots \circ x_u & 0 & \cdots & 0\\
0 & 0 & \cdots & 0\\
\vdots & \vdots & \ddots \vdots\\
0 & 0 & \cdots & 0
\end{pmatrix}
=
\begin{pmatrix}
x_1 & 0 & \cdots & 0\\
0 & 0 & \cdots & 0\\
\vdots & \vdots & \ddots \vdots\\
0 & 0 & \cdots & 0
\end{pmatrix}
\circ \cdots \circ
\begin{pmatrix}
x_u & 0 & \cdots & 0\\
0 & 0 & \cdots & 0\\
\vdots & \vdots & \ddots \vdots\\
0 & 0 & \cdots & 0
\end{pmatrix}
\]
we conclude that 
\[ \begin{pmatrix}
x_1 & 0 & \cdots & 0\\
0 & 0 & \cdots & 0\\
\vdots& \vdots & \ddots \vdots\\
0 & 0 & \cdots & 0
\end{pmatrix} 
\circ \cdots \circ
\begin{pmatrix}
x_v & 0 & \cdots & 0\\
0 & 0 & \cdots & 0\\
\vdots& \vdots & \ddots \vdots\\
0 & 0 & \cdots & 0
\end{pmatrix}\]
\[=
\begin{pmatrix}
x_1 \circ \cdots \circ x_v & 0 & \cdots & 0\\
0 & 0 & \cdots & 0\\
\vdots& \vdots & \ddots \vdots\\
0 & 0 & \cdots & 0
\end{pmatrix}
\subseteq M_m(P)\]\\
or 
\[ \begin{pmatrix}
x_{v+1} & 0 & \cdots & 0\\
0 & 0 & \cdots & 0\\
\vdots& \vdots & \ddots \vdots\\
0 & 0 & \cdots & 0
\end{pmatrix} 
\circ \cdots \circ
\begin{pmatrix}
x_u & 0 & \cdots & 0\\
0 & 0 & \cdots & 0\\
\vdots& \vdots & \ddots \vdots\\
0 & 0 & \cdots & 0
\end{pmatrix}\]
\[=
\begin{pmatrix}
x_{v+1} \circ \cdots \circ x_u & 0 & \cdots & 0\\
0 & 0 & \cdots & 0\\
\vdots& \vdots & \ddots \vdots\\
0 & 0 & \cdots & 0
\end{pmatrix}
\subseteq rad(M_m(P)).\]
In the first possibility, we get $x_1 \circ \cdots \circ x_v \subseteq P$. 
In the second possibility, there exists $n \in \mathbb{N}$ such that 
\[\begin{pmatrix}
(x_{v+1} \circ \cdots \circ x_u)^n & 0 & \cdots & 0\\
0 & 0 & \cdots & 0\\
\vdots& \vdots & \ddots \vdots\\
0 & 0 & \cdots & 0
\end{pmatrix}=
\begin{pmatrix}
x_{v+1} \circ \cdots \circ x_u & 0 & \cdots & 0\\
0 & 0 & \cdots & 0\\
\vdots& \vdots & \ddots \vdots\\
0 & 0 & \cdots & 0
\end{pmatrix}^n \subseteq M_n(P)\]

which implies $(x_{v+1} \circ \cdots \circ x_u)^n  \subseteq P$ and so $x_{v+1} \circ \cdots \circ x_u \subseteq rad(P)$. Thus we conclude that  $P$ is a $(u,v)$-absorbing primary hyperideal of $A$. 
\end{proof}
A non-empty subset $S$ of $A$ containing $1$  refers to a multiplicative closed subset (briefly, MCS) if $S$ is closed under the hypermultiplication  \cite{ameri}. Consider the set $(A \times S / \sim)$ of equivalence classes denoted by $S^{-1}A$ such that $(x,r) \sim (y,s)$ if and only if there exists $ t \in S $ with $ t \circ r \circ y=t \circ s \circ x$.
The equivalence class of $(x,r) \in A \times S$ is denoted by $\frac{x}{r}$. The triple $(S^{-1}A, \oplus, \odot)$ is a commutative multiplicative hyperring where

$\hspace{1cm}\frac{x }{r } \oplus \frac{y}{s}=\frac{r \circ y+s \circ x}{r  \circ s}=\{\frac{a+b}{c} \ \vert \ a \in r  \circ y , b \in s \circ x , c \in r \circ s\}$

$\hspace{1cm}\frac{x }{r } \odot \frac{y}{s}=\frac{x  \circ y}{r  \circ s}=\{\frac{a}{b} \ \vert \ a \in x \circ y, b \in r  \circ s\}$

The localization map $\pi: A \longrightarrow S^{-1}A$, defined by $a \mapsto \frac{a}{1}$, is a homomorphism of hyperrings. Furthermore, if $I$ is a hyperideal of $A$, then $S^{-1}P$ is a hyperideal of $S^{-1}P$ \cite{Mena}. Next, we discuss the relationship between $(u, v)$-absorbing primary hyperideals and their localizations.
\begin{theorem} \label{vvv}
Assume that $P$ is a $\mathcal{C}$-hyperideal of $A$, $S$ a MCS such that $P \cap S = \varnothing$ and $u,v \in \mathbb{N}$ with $u >v$. If $P$ is a $(u, v)$-absorbing primary hyperideal of $A$, then $S^{-1}P$ is a $(u-1, v-1 )$-absorbing primary hyperideal of $S^{-1}A$. 
\end{theorem}
\begin{proof}
Suppose that  $\frac{x_1}{r_1} \odot \cdots \odot \frac{x_{u-1}
}{r_{u-1}}=\frac{x_1 \circ \cdots \circ x_{u-1}}{r_1 \circ \cdots \circ r_{u-1}} \subseteq S^{-1}P$ for $r_1,\ldots,r_{u-1} \in S$ and  $x_1, \ldots , x_{u-1} \in A \backslash U(A)$ but $\frac{x_1}{r_1} \odot \cdots \odot \frac{x_{v-1}}{r_{v-1}} \nsubseteq S^{-1}Q$. Take any $r \in r_1 \circ \cdots \circ r_{u-1}$ and $x \in x_1 \circ \cdots \circ x_{u-1}$. Therefore $\frac{x}{r} \in \frac{x_1 \circ \cdots \circ x_{u-1}}{r_1 \circ \cdots \circ r_{u-1}}$ and so $\frac{x}{r}=\frac{x^{\prime}}{r^{\prime}}$ for some $r^{\prime} \in S$ and $x^{\prime} \in P$. Then, we conclude that  $t \circ x \circ r^{\prime}=t \circ x^{\prime} \circ r$ for some $t \in S$. This implies that $t \circ x \circ r^{\prime} \subseteq P$. Since $x \in x_1 \circ \cdots \circ x_{u-1}$, we obtain $t \circ x \circ r^{\prime} \subseteq t \circ x_1 \circ \cdots \circ x_{u-1} \circ r^{\prime}$. Since $P$ is a $\mathcal{C}$-hyperideal of $A$ and $t \circ x_1 \circ \cdots \circ x_{u-1} \circ r^{\prime} \cap P \neq \varnothing$, we get  $t \circ x_1 \circ \cdots \circ x_{u-1}\circ r^{\prime} \subseteq P$.  Take any $w \in t \circ r^{\prime}$. If $w \circ x_1 \circ \cdots \circ x_{v-1} \subseteq P$, then we get $\frac{x_1}{r_1} \odot \cdots \odot \frac{x_{v-1}}{r_{v-1}}=\frac{x_1 \circ \cdots \circ x_{v-1}}{r_1 \circ \cdots \circ r_{v-1}}= \frac{w \circ x_1 \circ \cdots \circ x_{v-1}}{w \circ r_1 \circ \cdots \circ r_{v-1}}\subseteq S^{-1}P$, a contradiction.  Since $P$ is a $(u, v)$-absorbing primary hyperideal of $A$, $w \circ x_1 \circ \cdots \circ x_{u-1} \subseteq P$ and $w \circ x_1 \circ \cdots \circ x_{v-1} \nsubseteq P$, we have $ x_v \circ \cdots \circ x_{u-1} \subseteq rad(P)$ which implies $\frac{x_v}{r_v} \odot \cdots \odot \frac{x_{u-1}}{r_{u-1}}=\frac{x_v \circ \cdots \circ x_{u-1}}{r_v \circ \cdots \circ r_{u-1}} \subseteq S^{-1}(rad(P))=rad(S^{-1}P)$. Consequently,  $S^{-1}P$ is a $(u-1,v-1)$-absorbing primary hyperideal of $S^{-1}A$. 
\end{proof}
\begin{theorem}
Let $P$ be a $\mathcal{C}$-hyperideal of $A$, $S$ be a MCS with $\Gamma \cap S = \varnothing$ where $\Gamma=\{a \in A \ \vert \ a \circ b \subseteq P \  \text{for some} \ b \in A \backslash P\}$ and $u,v \in \mathbb{N}$ with $u >v$. If $S^{-1}P$ is a $(u, v)$-absorbing primary hyperideal of $A$, then $P$ is a $(u,v )$-absorbing primary hyperideal of $ A$. 
\end{theorem}
\begin{proof}
Assume that $x_1 \circ \cdots \circ x_u \subseteq P$ for $x_1,\ldots,x_u \in A \backslash U(A)$. Then we have $\frac{x_1 \circ \cdots \circ x_u}{1 \circ \cdots \circ 1}= \frac{x_1}{1}\odot \cdots \odot \frac{x_u}{1}\subseteq S^{-1}P$. Since $S^{-1}P$ is a $(u, v)$-absorbing primary hyperideal of $A$, we conclude that $\frac{x_1 \circ \cdots \circ x_v}{1 \circ \cdots \circ 1}=\frac{x_1}{1}\odot \cdots \odot \frac{x_v}{1}  \subseteq S^{-1}P$ or $\frac{x_{v+1} \circ \cdots \circ x_u}{1 \circ \cdots \circ 1}=\frac{x_{v+1}}{1}\odot \cdots \odot \frac{x_u}{1}  \subseteq rad(S^{-1}P)=S^{-1}rad(P)$. In the first case, we have $\frac{a}{1} \in S^{-1}P$ for some $a \in x_1 \circ \cdots \circ x_v$. Therefore there exists $p \in P$ and $t \in S$ such that $\frac{a}{1}=\frac{p}{t}$ and so $s \circ t \circ a=s \circ p \circ 1$ for some $s \in S$ which means $s \circ t \circ a \subseteq P$. Take any $r \in s \circ t$.  Since $\Gamma \cap S=\varnothing$ and $r \circ a \subseteq P$, we have $a \in P$. Since $P$ is a $\mathcal{C}$-hyperideal of $A$ and $(x_1 \circ \cdots \circ x_v) \cap P \neq \varnothing$, we obtain $x_1 \circ \cdots \circ x_v \subseteq  P$. In the second case, we get $\frac{b}{1} \in S^{-1}rad(P)$ for some $b \in x_{v+1} \circ \cdots \circ x_u$. Therefore there exists $q \in rad(P)$ and $t \in S$ such that $\frac{b}{1}=\frac{q}{t}$ and so $s \circ t \circ b=s \circ q \circ 1$ for some $s \in S$ which means $s \circ t \circ b \subseteq rad(P)$. Assume that $c \in r \circ b$ such that $r \in s \circ t$. Then there exists $n \in \mathbb{N}$ such that $c^n \subseteq P$. since $P$ is a $\mathcal{C}$-hyperideal of $A$ and $ (r \circ x_{v+1} \circ \cdots \circ x_u)^n \cap P \neq \varnothing$, we get $(r \circ x_{v+1} \circ \cdots \circ x_u)^n \subseteq P$. Let $\alpha \in r^n$ and $\beta \in  (x_{v+1} \circ \cdots \circ x_u)^n$. Since $\Gamma \cap S=\varnothing$, $\alpha \circ \beta \subseteq P$ and $\alpha \notin \Gamma$, we have $\beta \in P$ and so $(x_{v+1} \circ \cdots \circ x_u)^n \subseteq P$ which means $x_{v+1} \circ \cdots \circ x_u \subseteq rad(P)$. Thus, $P$ is a $(u,v )$-absorbing primary hyperideal of $ A$.
\end{proof}

\section{conclusion}
In this paper, we introduced and investigated  an expansion of 1-absorbing primary hyperideals in multiplicative hyperrings called $(u,v)$-absorbing primary  hyperideals  where $u,v \in \mathbb{Z}$ with $u>v$. We gave  several specific results explaining this new structure.  We indicated that the  concepts of $(u,v)$-absorbing prime  hyperideals and $(u,v)$-absorbing primary  hyperideals are different, although every $(u,v)$-absorbing prime  hyperideal is a $(u,v)$-absorbing primary  hyperideal. We analyzed when a product of two hyperideals is a $(u,v)$-absorbing primary hyperideal. Furthermore, it was examined when $(P:x)$ is a $(u,v)$-absorbing primary hyperideal. We presented a condition by which the intersection of some $(u,v)$-absorbing primary hyperideals is a $(u,v)$-absorbing primary hyperideal and then gave an example showing this condition is crucial. We concluded that a $\mathcal{C}$-hyperideal in a Dedekind hyperdomain is a $(u,v)$-absorbing primary hyperideal if and
only if it$^,$s radical is a prime hyperideal.
 Also, we investigated when    a $\mathcal{C}$-hyperideal in a divided multiplicative hyperring $A$ is a $(u,v)$-absorbing primary hyperideal. Moreover, we studied the stability of   $(u,v)$-prime hyperideals with respect to localization. The study can be continued for other classes of hyperstructures.
\section{Future work}
In \cite{Akray}, Akray and Anjuman proposed the notion of $v$-absorbing $I$-primary hyperideals in a multiplicative hyperring. As future work, we will study the concept $(u,v)$-absorbing $I$-primary  hyperideal as a generalization of $(u,v)$-absorbing primary  hyperideals. 
\begin{definition}
Let $P$ be a proper hyperideal of $A$ and $u,v \in \mathbb{N}$ with $u >v$. For fixed proper hyperideal $I$ of $A$, we say that $P$ is a  $(u,v)$-absorbing $I$-primary  hyperideal  if     $x_1 \circ \cdots \circ x_u \subseteq P \backslash IP$ for $x_1,\ldots, x_u \in A \backslash U(A)$, then either $x_1 \circ \cdots \circ x_v \subseteq P$ or $x_{v+1} \circ \cdots \circ x_u \subseteq rad(P)$.
\end{definition}


\end{document}